\documentclass[10pt,draft]{amsart} 
\usepackage{amsmath,amssymb,amscd}

\newcommand{\C}{\mathbb{C}}
\newcommand{\tensor}{\mathop{\otimes}}
\newcommand{\dirsum}{\mathop{\oplus}}

\newcommand{\Sym}{S}
\newcommand{\g}{\mathfrak{g}}
\newcommand{\h}{\mathfrak{h}}
\newcommand{\n}{\mathfrak{n}}

\newcommand{\Sl}{\mathfrak{sl}}

\newcommand{\sgn}{\mathrm{sgn}}
\newcommand{\exterior}{{\textstyle\bigwedge}}
\newcommand{\rank}{r}

\newcommand{\dash}{\nobreakdash-\hspace{0pt}}
\newcommand{\End}{\qopname\relax o{End}}
\newcommand{\ad}{\mathop{ad}}

\newcommand{\K}{\Phi}
\newcommand{\J}{J}
\newcommand{\Cl}{{\mathit{C}\!\ell}}

\newcommand{\rmatr}{\mathfrak{r}}
\newcommand{\zerobar}{{\bar 0}}
\newcommand{\onebar}{{\bar 1}}

\newcommand{\pr}{\mathit{pr}}
\newcommand{\id}{\mathrm{id}}

\theoremstyle{plain}
\newtheorem{theorem}{Theorem}
\newtheorem{proposition}[theorem]{Proposition}
\newtheorem{lemma}[theorem]{Lemma}
\newtheorem{corollary}[theorem]{Corollary}
\theoremstyle{definition}

\theoremstyle{remark}
\newtheorem{remark}[theorem]{Remark}
\numberwithin{theorem}{section}

\advance\oddsidemargin by-0.5in
\advance\evensidemargin by-0.5in
\advance\textwidth by1.0in
\advance\topmargin by-0.35in
\advance\textheight by0.75in

\begin{document}
\title[Harish-Chandra isomorphism for Clifford algebras]{The
  Harish-Chandra isomorphism for Clifford algebras} 
\author{Yuri Bazlov}
\address{Mathematics Institute, University of Warwick, Coventry CV4
  7AL, United Kingdom}
\email{y.bazlov@warwick.ac.uk}
\begin{abstract}
We study an analogue of the Harish-Chandra homomorphism where the
universal enveloping algebra $U(\g)$ is replaced by the Clifford
algebra, $\Cl(\g)$, of a semisimple Lie algebra $\g$. 
Two main goals are achieved. First, we prove that there is a
Harish-Chandra type isomorphism between the subalgebra of $\g$\dash
invariants in $\Cl(\g)$ and the Clifford algebra $\Cl(\h)$ of the
Cartan subalgebra of $\g$. Second, the Cartan subalgebra $\h$ is
identified, via this isomorphism, with the graded space of the so-called
primitive skew\dash symmetric invariants of $\g$. The grading leads to a
distinguished orthogonal basis of $\h$, which turns out to be induced
from the Langlands dual Lie algebra $\g^\vee$ via the action of its
principal three\dash dimensional subalgebra.
This settles a conjecture of Kostant.  
\end{abstract}
\maketitle

\setcounter{tocdepth}{1}
\section*{Introduction}

Introduced by Harish\dash Chandra 
more than half a century ago, 
the Harish\dash Chandra homomorphism %
is of utmost significance in 
representation theory of semisimple Lie groups and algebras; 
character theory is one of the areas where it  plays a key role.
Recall that, given a complex semisimple Lie
algebra $\g$ and its Cartan subalgebra $\h$, 
the Harish\dash Chandra homomorphism is a one-to-one algebra
map between the centre $Z(\g)$ of the universal enveloping algebra
$U(\g)$ of $\g$ and the algebra $S(\h)^W$ of (translated) Weyl group
$W$ invariant polynomial functions on the space $\h^*$. 
Because $W$ is a finite reflection group, by the well-known
Chevalley\dash Shephard\dash Todd theorem 
$S(\h)^W$ is a polynomial algebra. The Harish\dash Chandra map
thus establishes the 
polynomiality of the algebra $Z(\g)$ and identifies the characters of
$Z(\g)$ with $W$\dash orbits in $\h^*$.  

The subject of the present paper is a natural analogue of the
Harish\dash Chandra homomorphism where $U(\g)$ is replaced with the
Clifford algebra, $\Cl(\g)$, of a semisimple Lie algebra $\g$.
This analogue is based on a remarkably easy algebraic construction,
which underlines the ``classical'' Harish\dash Chandra homomorphism but 
applies to a class of algebras much wider than that of universal
enveloping algebras $U(\g)$. Let $A$ be an associative algebra that
factorises as $A_- \tensor A_0 \tensor A_+$,  
where $A_\pm$ and $A_0$ are subalgebras in $A$ and the tensor product is
realised by the multiplication in $A$; we stress that $A_-$, $A_0$ and
$A_+$ need not commute in $A$.
Some further conditions on this
factorisation guarantee
that there exists a projection (in general, not an algebra map) 
$\pr=\varepsilon_-\tensor \id \tensor \varepsilon_+$ of $A$ onto its
subalgebra $A_0$. (General details are discussed in \cite{BB}.)

In the case $A=U(\g)$, the  standard triangular
decomposition  $\g=\n_-\dirsum \h\dirsum \n_+$ of a semisimple Lie
algebra $\g$ gives rise to a factorisation as above 
with $A_\pm = U(\n_\pm)$ and $A_0 = U(\h) = S(\h)$ the polynomial
algebra. The Harish\dash Chandra map is
obtained by restricting $\pr$ to $Z(\g)=U(\g)^\g$, the $\ad \g$\dash
invariants in $U(\g)$. 
(This method also works for the quantised universal 
enveloping algebra $U_\hbar(\g)$ \cite{T}.) 
The same approach is used in \cite{WH} to give a completely algebraic  
construction of the homomorphism  $\mathcal D(\g)^\g \to \mathcal
D(\h)^W$, also due to Harish-Chandra (here
$\mathcal D(\cdot)$ stands for polynomial differential operators).

The Clifford algebra $\Cl(\g)$ factorises as a product of three of its
subalgebras, $\Cl(\n_-)$, $\Cl(\h)$ and $\Cl(\n_+)$, 
therefore admitting a Harish-Chandra map $\K\colon \Cl(\g)\to\Cl(\h)$. 
Studying the map~$\K$ involves structural theory of the
Clifford algebra $\Cl(\g)$, based mostly on results of Kostant. In the 
present paper, we focus on the restriction of the
Harish-Chandra map $\K$ to the subalgebra
$\J=\Cl(\g)^\g$ of $\g$\dash invariants, which in the Clifford
algebra plays a role  similar to that of the centre $Z(\g)=U(\g)^\g$ in
$U(\g)$.  Indeed, Kostant's ``separation of
variables'' result \cite{Ko} 
states that $\Cl(\g)=E\tensor \J$ is a free module
over the algebra $\J$, and 
$\J$ is further described as a Clifford algebra $\Cl(P)$
of a remarkable space $P$ of the so-called primitive invariants. 

Our first main result, Theorem~\ref{thm:main1}, shows that $\K$
restricts to an isomorphism $\J\xrightarrow{\sim}\Cl(\h)$. Comparing
this to the universal enveloping algebra, it is worth
noting that the action of the Weyl group $W$ on $\Cl(\h)$ is
conspicuously absent from the picture. We then refine the isomorphism
result by
showing that the Harish-Chandra map $\K$ identifies the space $P$ of
primitive invariants with the Cartan subalgebra $\h$ of $\g$. 

The above immediately raises the next question: the space $P$ is naturally
graded, via its inclusion in the exterior algebra $\exterior \g$ and
its identification with primitive homology classes in the homology of
$\g$; what is the grading induced on the Cartan subalgebra by the
Harish-Chandra map? It turns out that the answer involves 
the so-called Langlands dual Lie algebra $\g^\vee$ of $\g$. This is
a complex semisimple Lie algebra with a root system dual to that of
$\g$. There is a canonical copy of $\Sl_2$ inside $\g^\vee$, and its
adjoint action on $\g^\vee$ splits $\g^\vee$ into a direct sum of
$\Sl_2$\dash submodules (``strings''). Intersections of these
``strings'' with the Cartan subalgebra $\h$, which is viewed as shared
between $\g$ and $\g^\vee$, define the graded components of $\h$. This
is established in Theorem~\ref{thm:main2} which is the second and
final main result of the paper. In particular, Theorem~\ref{thm:main2}
confirms a conjecture made by Kostant \cite{KC}. 

When $\g$ is a simple Lie algebra, the graded components in $\h$ typically are
one\dash dimensional, giving rise to a distinguished basis of
$\h$. (The only deviation from this occurs in a simple Lie algebra with
Dynkin diagram of type $D$ on an even number of nodes.) 
This basis is orthogonal with respect to the Killing form and
contains $\rho$, half the sum of positive roots of $\g$; we refer to
this basis of $\h$ as the (Langlands dual) principal basis of the Cartan
subalgebra. 

\subsection*{Acknowledgments}

Section~\ref{sect:iso} of the present paper is based on a 
revision of work done in my PhD thesis \cite{thesis}, which was completed 
with the guidance of Anthony Joseph. 
I am grateful to Bertram Kostant for conversations on the subject of
this paper; having seen partial results of \cite{thesis}, 
he suggested a conjectural
link to principal TDS bases, proved in the present paper. 
I thank Anton Alekseev for 
providing helpful ideas for the proof of Theorem~\ref{thm:main2}. 
These useful ideas and calculations are also made clear
in a recent preprint \cite{Rohr} by R.~P.~Rohr. 
I am grateful to Victor Ginzburg and Eckhard Meinrenken for their comments
on earlier versions of this work and to David Kazhdan for a stimulating
discussion. 
 
\tableofcontents


\section{The classical Chevalley projection and 
Harish-Chandra map}
\label{sect:firstsect}

We start by recalling the Chevalley projection map and the Harish\dash
Chandra map associated with a semisimple Lie algebra $\g$. This will
serve as an introduction to the subsequent treatment of 
the analogues of these for skew\dash symmetric tensors.

\subsection{Invariant symmetric tensors and the Chevalley projection}

Let $\g$ be a semisimple Lie algebra of rank $\rank$ over the complex
field $\C$. 
Let a Cartan subalgebra $\h$ and a Borel subalgebra $\mathfrak b$ of
$\g$, such that $\h\subset\mathfrak b$,  be fixed. This partitions
the root system of $\g$ into positive and negative parts, and 
gives rise to the triangular decomposition
\begin{equation*}
	\g = \n_- \dirsum \h \dirsum \n_+
\end{equation*}
of $\g$, where $\n_-$ and $\n_+$ are subspaces of $\g$ 
spanned by root vectors corresponding 
to negative and positive roots, respectively.
By $\Sym(\g)$ we denote the algebra of symmetric tensors over  $\g$,
which is the same as the algebra of polynomial functions on the space $\g^*$. 
It is graded by degree: $\Sym(\g)
=\bigoplus_{n=0}^\infty \Sym^n(\g)$, and we identify $\g$ with
$S^1(\g)$. The action of $\g$ on $\Sym(\g)$ is 
extended from the adjoint action of $\g$ on $\g=S^1(\g)$ by
derivations of degree $0$; we denote the action of $x\in \g$ by $ \ad
x\in \End \Sym(\g)$.   
We refer to the set 
$$
                      \J_S=\Sym(\g)^\g = \{f\in \Sym(\g) \mid (\ad x
                      )f=0 \ \forall x\in \g\}
$$ 
as the space of symmetric $\g$\dash invariants.  
Note that $\J_S$ is a graded subalgebra of $\Sym(\g)$.
It is obvious that the algebra  $\Sym(\g)$ has triangular
factorisation, 
$$
\Sym(\g) \cong \Sym(\n_-) \tensor \Sym(\h) \tensor \Sym(\n_+),
$$
into three subalgebras generated by $\n_-$, $\h$, $\n_+$,
respectively. Denote by $\varepsilon$ the character of an algebra of
polynomials given by the evaluation of a polynomial at zero. 
The homomorphism 
$$
\Psi_0 = \varepsilon \tensor \id \tensor \varepsilon \colon
\Sym(\g)\to \Sym(\h) 
$$
of commutative algebras 
is what is typically called  the Chevalley projection map.
By a classical result of Chevalley, see \cite[Theorem 7.3.7]{Dix}, 
the restriction of $\Psi_0$ to $\J_S$ is a graded
algebra isomorphism
\begin{equation*}
	\Psi_0\colon \J_S \xrightarrow{\sim} \Sym(\h)^W.
\end{equation*}
Here $\Sym(\h)^W$ denotes symmetric tensors over $\h$  
invariant under the action of the Weyl group $W$ of $\g$ (this action
is extended from $\h$ to $\Sym(\h)$). 

\subsection{The classical Harish-Chandra map}
\label{subsect:cHC}

Let $U(\g)$ be the universal enveloping algebra of $\g$, and 
$$
    \J_U = U(\g)^{\g}
$$ 
be its centre, which is the ring of invariants
of the adjoint representation of $\g$ in $U(\g)$.
The Poincar\'e\dash Birkhoff\dash Witt symmetrisation map 
$$
\beta\colon \Sym(\g) \to U(\g)
$$ 
is a
$\g$\dash module isomorphism between $\Sym(\g)$ and $U(\g)$, where
$\beta(x_1x_2\dots x_n)$ for $x_i\in \g$ is defined as 
$\frac{1}{n!}$ times the sum of $x_{\pi(1)}x_{\pi(2)}\dots
x_{\pi(n)}\in U(\g)$ over all permutations $\pi$ in $n$ letters. 
In particular, $\beta$ identifies $\J_S$ and $\J_U$ as linear spaces.
Note that $\beta$ is not an algebra isomorphism between $\J_S$ and
$\J_U$. (An algebra isomorphism $\J_S\xrightarrow{\sim}\J_U$, known as
the Duflo map, was explicitly constructed  in \cite{D}.) 
The Poincar\'e\dash Birkhoff\dash Witt theorem for $U(\g)$ implies
the triangular factorisation 
$$
    U(\g) = U(\n_-) \tensor U(\h) \tensor U(\n_+)
$$
of the algebra $U(\g)$ into the subalgebras generated by $\n_-$, $\h$,
$\n_+$, respectively. Let $\varepsilon_-\colon U(\n_-)\to \C$ be
the algebra homomorphism defined by $\varepsilon_-(x)=0$ for $x\in
\n_-$, and let $\varepsilon_+\colon U(\n_+)\to \C$ be defined
similarly. We will refer to the map  
$$
    \Psi = \varepsilon_-\tensor \id \tensor \varepsilon_+ \colon U(\g)
    \to U(\h) 
$$
as the Harish\dash Chandra map, slightly abusing the terminology. 
As $\h$ is an Abelian Lie algebra, $U(\h)$ is identified with the
polynomial algebra $\Sym(\h)$. 
The map $\Psi$ is not an algebra homomorphism, but its restriction to 
the subalgebra $U(\g)^\h = \{f\in U(\g) \mid (\ad h)f=0 \ \forall h\in
\h\}$ is. Further restricting $\Psi$ to $\J_U$ which is a subalgebra
of $U(\g)^\h$, one obtains the isomorphism  
$$
       \Psi \colon \J_U \xrightarrow{\sim} \Sym(\h)^{W_\cdot}
$$
between the centre of $U(\g)$ and the ring $\Sym(\h)^{W_\cdot}$ 
of symmetric tensors over $\h$ invariant under the shifted action of
$W$. (The shifted action of $w\in W$ is defined on $\lambda\in \h^*$ by 
$w.\lambda=w(\lambda+\rho)-\rho$ where $\rho\in \h^*$ 
is the half sum of positive
roots of $\g$, and hence on $\Sym(\h)$ which is viewed 
as the algebra of polynomial functions on
$\h^*$.) This isomorphism is due to Harish\dash Chandra; 
see \cite{HC}, \cite[Theorem 7.4.5]{Dix}.

\subsection{An analogue of the Chevalley projection for $\exterior\g$}

We are going to consider $\exterior\g$, the exterior algebra of 
$\g$, as a ``skew\dash symmetric analogue'' of
$\Sym(\g)$. Furthermore, 
the universal enveloping algebra $U(\g)$ --- a deformation of $S(\g)$
--- will be replaced by
$\Cl(\g)$, the Clifford algebra of $\g$, which is a deformation of
$\exterior \g$. 
We will now discuss the analogue of the Chevalley projection 
$\Psi_0$ 
(and later, of the Harish\dash Chandra map $\Psi$) in the skew\dash
symmetric situation.


%
%

The action of $\g$ on the finite\dash dimensional algebra
$\exterior\g=\bigoplus_{n=0}^{\dim \g} \exterior^n \g$ is an extension
of the adjoint action of $\g$ as a derivation of degree $0$. Let us
denote the action of 
$x\in\g$ by $\theta(x)\in\End(\exterior\g)$; that is, 
$\theta(x)y=[x,y]$ where $y\in \wedge^1 \g = \g$, and 
$\theta(x)(u\wedge v) = (\theta(x)u)\wedge v + u\wedge \theta(x)v$ for
$u,v\in\exterior\g$. The subspace
$$
     \J =(\exterior\g)^{\g} = \{u\in \exterior\g \mid \theta(x)u
     =0 \ \forall x\in \g\},
$$
of $\theta(\g)$ invariants (the invariant
skew\dash symmetric tensors over $\g$) 
is a graded $\wedge$\dash subalgebra of $\exterior\g$.
The triangular factorisation 
$$
   \exterior \g = \exterior \n_- \tensor \exterior \h \tensor
   \exterior \n_+
$$
and the ``augmentation maps'' $\varepsilon_\pm \colon \exterior\n_\pm
\to \C$ which are algebra homomorphisms  uniquely defined by
$\varepsilon_\pm(\n_\pm)=0$,  give rise to a degree\dash preserving 
projection map 
$$
    \K_0=\varepsilon_-\tensor \id \tensor \varepsilon_+ \colon
    \exterior\g \to \exterior\h\ .
$$
Let us consider the 
restriction of $\K_0$ to $\J$. 
In contrast to the Chevalley projection map for symmetric tensors, 
this restriction fails to be one-to-one: one has 
\begin{equation*}
	\K_0(\J)=\exterior^0\h=\C.
\end{equation*}
Indeed, one shows that the $\K_0$\dash image of $\J$ must
lie in the $W$\dash invariants in $\exterior \h$. However, the fixed
points of $W$ in $\exterior \h$ are just the one\dash dimensional space
$\exterior^0\h$; see \cite[Section 5.1]{GP}.

Although the injectivity of the Chevalley projection $\K_0$ on $\exterior \g$
fails so miserably, passing to its counterpart $\K$ (a Clifford algebra
version of the Harish\dash Chandra map) rectifies the situation, as we
will discover in due course.

\section{Clifford algebras} 

In this Section, we recall some basics on  Clifford algebras
associated to quadratic forms on complex vector spaces, 
define the algebra $\Cl(\g)$, and introduce the Clifford algebra
version of the Harish-Chandra map.  

\subsection{Identification of $\Cl(V)$ with $\exterior V$}
\label{subsect:sigma}

Let $V$ be a finite\dash dimensional vector space over $\C$, equip\-ped
with a symmetric bilinear form $(\,,\,)$. The Clifford algebra
$\Cl(V)$ of $V$ is the quotient of the full tensor algebra $T(V)$
modulo the two\dash sided ideal generated by $\{x\otimes x-(x,x) \mid
x\in V \}$. 
We will denote the Clifford product of $u,v\in\Cl(V)$ by $u\cdot v$. 
An isomorphic image of the space $V=T^1(V)$ is contained in $\Cl(V)$;  
one has $xy+yx=2(x,y)$ for $x,y\in V$. 

A convenient point of view that we adhere to in the present paper is that 
the Clifford algebra $\Cl(V)$ has the same underlying linear space as
$\exterior V$, but the Clifford product is a deformation of the
exterior product. Explicitly, following \cite{Ko}, for $x\in V$ 
let the operator $\iota(x)\colon \exterior^1 V \to \exterior^0 V$ 
be defined by $\iota(x)y = (x,y)$ 
where $y\in V$. Extend $\iota(x)$ to a superderivation of $\exterior
V$ of degree $-1$, i.e., by the rule  
$$
\iota(x)(u\wedge v)=(\iota(x)u)\wedge v
+(-1)^{|u|} u \wedge\iota(x)v, \qquad u\in \exterior^{|u|}V, \ v\in
\exterior V. 
$$
We refer to $\iota(x)\colon \exterior V \to \exterior V$ as the
contraction operator associated to $x\in V$.  
Now define the operator $\gamma(x)\in \End(\exterior V)$ by
$\gamma(x)u = x\wedge u + \iota(x)u$, $u\in\exterior\g$.  
The superderivation property of $\iota(x)$ and the fact that
$\iota(x)^2=0$ and $x\wedge x=0$ imply that $\gamma(x)^2$ 
is multiplication by the scalar $(x,x)$. Therefore, $\gamma$ extends 
to a homomorphism
$\gamma\colon \Cl(V) \to \End(\exterior V)$. The linear map 
$$
    \sigma\colon \Cl(V)\to \exterior V, \qquad \sigma(u) = \gamma(u)1,
$$
is bijective, cf.\ \cite[Theorem II.1.6]{Chev}. 
It is this map $\sigma$ that is used to identify the underlying linear
space of $\Cl(V)$ with $\exterior V$. 

\begin{remark}
\label{rem:braided}
The symmetric algebra $S(V)$ and the exterior algebra $\exterior V$ of
a vector space $V$ are the simplest examples of the so\dash called Nichols
algebras (terminology introduced by Andruskiewitsch and Schneider). 
Nichols algebras are Hopf algebras in a braided category;
see \cite{N,W,AS,B} for background. In particular, Nichols algebra
$\mathcal B(V)$ of a braided space $V$ (a space equipped with an
invertible operator $c\in \End(V\tensor V)$ satisfying the quantum Yang\dash
Baxter equation) has a set of braided derivations $\partial_\xi$,
indexed by $\xi\in V^*$, that satisfy the braided Leibniz rule
inferred from the braiding $c$. 
One takes the braiding $c(x\tensor y) = y\tensor x$, respectively
$c(x\tensor y)=-y\tensor x$, to obtain the
Nichols algebra $S(V)$, respectively $\exterior V$, of $V$.
Via the map $V \to V^*$ given by the form $(\,,\,)$, the 
partial derivative $\frac{\partial}{\partial x}$ and the contraction operator
$\iota(x)$, $x\in V$,
are braided derivations of $S(V)$ and $\exterior V$, respectively. 
\end{remark}

We also remark that the inverse to the map $\sigma\colon
\Cl(V)\to\exterior V$ can be written, in
a completely different fashion, as the skew\dash
symmetrisation map 
$$
    \sigma^{-1} = \beta_\wedge \colon \exterior V \to \Cl(V), 
\qquad
    \beta_\wedge(x_1\wedge \dots \wedge x_n)= \frac{1}{n!} \sum_\pi 
     (\sgn\ \pi)x_{\pi(1)}x_{\pi(2)} \dots x_{\pi(n)},
$$
where $x_i\in V$ and the sum on the right is over all permutations
$\pi$ of the indices $1,\dots,n$. 

\subsection{Algebras $\Cl_\hbar(V)$}

To emphasise our earlier point that the Clifford product on $\exterior
V$ is a deformation of the wedge product, and for later use in calculations,
we introduce a complex\dash valued deformation
parameter $\hbar$. To each value of
$\hbar$ we associate a bilinear form $(x,y)_\hbar:=\hbar\cdot (x,y)$ on~$V$. Denote by $\Cl_\hbar(V)$ the Clifford algebra of the form
$(\,,\,)_\hbar$. Observe that, for $\hbar\ne 0$, the algebra
$\Cl_\hbar(V)$ is isomorphic to the original Clifford algebra 
$\Cl(V)=\Cl_\hbar(V)|_{\hbar=1}$, because the linear map $V \to
V$, $x\mapsto \hbar^{1/2}x$, extends to an 
algebra isomorphism $\Cl_\hbar(V)\to\Cl(V)$. On the other hand, 
$\Cl_\hbar(V)|_{\hbar=0}$ coincides with the exterior algebra
$\exterior V$ and is not isomorphic to $\Cl(V)$ unless the form
$(\,,\,)$ is identically zero. 

This construction gives rise to a family of Clifford products
$\{\cdot_\hbar \mid \hbar\in\C\}$ on the space $\exterior V$. The way
the product $a\cdot_\hbar b$ depends on $\hbar$ is described in 
\begin{lemma}
\label{lem:taylor}
If $a\in \exterior^i V$, $b\in\exterior^j V$ are homogeneous
elements in the exterior algebra of $V$, there exist 
$u_{i+j-2s}\in \exterior^{i+j-2s}V$ for $s=1,2,\dots,\lfloor
\frac{i+j}{2}\rfloor$ such that
$$
    a\cdot_\hbar b = a\wedge b + \hbar^{\phantom{2}}u_{i+j-2} + \hbar^2 u_{i+j-4} +
    \dots = a\wedge b + \sum_{1\le s \le (i+j)/2} \hbar^s u_{i+j-2s}.
$$ 
\end{lemma}
\begin{proof}
The statement is proved by induction in $i$, the degree of $a$. If
$i=0$, put $u_{0+j-2s}=0$ for all $s$. If $i=1$ and $a=x\in V$, then 
$a\cdot_\hbar b=x\wedge b + \hbar \, \iota(x) b$. In this case, put 
$u_{1+j-2}= \iota(x) b$ and $u_{1+j-2s}=0$ for all $s>1$. 

To prove the assertion for $i\ge 2$, one may assume
that $a=x\wedge a'$ for some $x\in V$ and $a'\in\exterior^{i-1} V$. Put 
$a''=\hbar \, \iota(x) a'$ so that the element $a''$ is homogeneous of
degree $i-2$ in $\exterior V$. 
Then $a=x\cdot_\hbar a' - a''$, hence $a\cdot_\hbar b=x\cdot_\hbar
(a'\cdot_\hbar b) - a''\cdot_\hbar b$. By the induction hypothesis, 
the Clifford products $a'\cdot_\hbar b$ and $a''\cdot_\hbar b$ have
required expansions in $\exterior V$.
It only remains to apply the $i=1$ case to the product $x\cdot_\hbar
(a'\cdot_\hbar b)$ and to collect the terms, which gives the required
expansion for $a\cdot_\hbar b$.
\end{proof}

\subsection{The superalgebra structure on $\Cl(V)$}

Clearly, the Clifford product on the exterior algebra $\exterior V$
does not respect the grading on $\exterior V$, and $\Cl(V)$ is not a
graded algebra. It is, however, easy to see (and is apparent from
Lemma~\ref{lem:taylor}) that $\Cl(V)$ is still a superalgebra:
$$
    \Cl(V) = \Cl^\zerobar(V) \dirsum \Cl^\onebar(V),
$$
with $\Cl^i(V)=\sum_{n\ge 0}\exterior^{2n+i}V$
for $i=\zerobar,\onebar$ (residues modulo $2$).
For later use, we observe the fact that the contraction operators
$\iota(x)\colon \exterior V \to \exterior V$ 
are superderivations with respect to the Clifford
multiplication:
\begin{lemma}
\label{lem:superderiv}
For any $i\in \{\zerobar,\onebar\}$, 
$x\in V$, $u\in \Cl^i(V)$ and $v\in \Cl(V)$,  
$$
    \iota(x)(u\cdot v) = (\iota(x)u)\cdot v + (-1)^i u\cdot \iota(x)v. 
$$
\end{lemma}
\begin{proof}
We have to show that $\iota(x)$ supercommutes with $\gamma(u)\colon
\exterior V \to \exterior V$, the operator of the left Clifford
multiplication by $u$. Since $\gamma(u)$ is in the subalgebra of
$\End(\exterior V)$ generated by $\gamma(y)$, $y\in V$, 
it is enough to show that $\iota(x)$ supercommutes with $\gamma(y)$.
Write $\gamma(y)=(y\wedge \cdot )+\iota(y)$. Now, $\iota(x)$
supercommutes with the operator $y\wedge \cdot$ of the left exterior
multiplication by $y$, because $\iota(x)$ is a superderivation of the
wedge product. Finally, to show that $\iota(x)$ supercommutes with
$\iota(y)$, observe that, by a general fact about superderivations,
the supercommutator 
$\iota(x)\iota(y)+\iota(y)\iota(x)$ must be an even superderivation of the
wedge product; but it obviously vanishes on $V$, hence is identically
zero on~$\exterior V$.  
\end{proof}

\subsection{The Clifford algebra $\Cl(\g)$}

We are interested in the case when $V=\g$ is a semisimple Lie algebra.
We fix $(\,,\,)$ to be a non\dash degenerate $\ad$\dash invariant symmetric
bilinear form on $\g$. For example, $(\,,\,)$ may be the Killing form,
or be proportional to the Killing form with a non\dash zero coefficient. (If
$\g$ is simple, there are no other options.)  
We denote by $\Cl(\g)$ the Clifford algebra of $\g$ with respect to
the form $(\,,\,)$. 

Recall that by $\theta(g)$ is denoted the adjoint action of $g\in \g$
on the exterior algebra $\exterior \g$; one has 
$\theta(g)(x\wedge u)-x\wedge \theta(g)u = [g,x]\wedge u$ for
all $g,x\in \g$ and $u\in\exterior\g$. 
Furthermore, it is easy to see that the $\ad$\dash invariance of the
form $(\,,\,)$  
implies $\theta(g)\iota(x) - \iota(x)\theta(g) = \iota([g,x])$. Thus,
$\theta(g)$ is an (even) derivation of the Clifford product. 
It immediately follows that $\g$\dash invariants
$\J=(\exterior\g)^\g$
form a Clifford subalgebra in $\Cl(\g)$. Recall that $J$ is also a
wedge\dash subalgebra in $\exterior \g$. 
We will  elaborate on these two  non\dash isomorphic algebra
structures on $J$ in the next Section.

\subsection{The Harish-Chandra map $\K$ for $\Cl(\g)$}

The central object of the paper is 
the following analogue of the Harish-Chandra map, defined for the
Clifford algebra $\Cl(\g)$.
Observe that $\Cl(\g)$, like all the algebras considered so far,
factorises into its subalgebras, generated by the direct summands 
$\n_-$,  $\h$ and $\n_+$ in the triangular decomposition of
$\g$. These subalgebras are themselves Clifford algebras that
correspond to the restriction of the form $(\,,\,)$ on the
respective subspaces of $\g$: 
$$
   \Cl(\g) = \Cl(\n_-) \tensor \Cl(\h) \tensor \Cl(\n_+),
$$
where the tensor product is realised by the Clifford multiplication 
(note that the tensorands do not commute with respect to Clifford
multiplication). 
The subalgebras $\Cl(\n_\pm)$ are supercommutative and are 
isomorphic to the exterior algebras $\exterior \n_\pm$, 
because the restriction of
$(\,,\,)$ to $\n_-$ (respectively to $\n_+$) is necessarily zero. 
The restriction of $(\,,\,)$ to the Cartan subalgebra $\h$ is non\dash
degenerate, thus $\Cl(\h)$ is a simple algebra or a direct sum of two simple
algebras, much like the bigger algebra $\Cl(\g)$.

In line with all the previous definitions of Harish\dash Chandra
type maps, introduce the Harish\dash Chandra map for $\Cl(\g)$ by  
\begin{equation*}
	\K=\varepsilon_-\tensor \id \tensor \varepsilon_+ \colon
	\Cl(\g) \to \Cl(\h), 
\end{equation*}
where $\varepsilon_\pm \colon \Cl(\n_\pm) = \exterior\n_\pm \to \C$ are
the augmentation maps as above. Similar to the $U(\g)$ situation, the
Harish\dash 
Chandra map $\K$ is not a homomorphism of algebras, but its
restriction to $\h$\dash invariants is: 
\begin{lemma}
\label{lem:homo}
The restriction of $\K$ to the subalgebra $\Cl(\g)^\h=\{u\in \Cl(\g)
\mid \theta(\h)u=0\}$ is a superalgebra homomorphism between $\Cl(\g)^\h$
and $\Cl(\h)$.  
\end{lemma}
\begin{proof}
First of all, $\K$ is a map of superspaces: the triangular
factorisation $\Cl(\g)=\Cl(\n_-)\tensor$ $\Cl(\h) \tensor$ $\Cl(\n_+)$ is
compatible with the superspace structure on all tensorands, and
moreover, the maps
$\varepsilon_\pm \colon \Cl(\n_\pm) \to \C$ are superspace maps (where
$\C$ is a one\dash dimensional even space), and $\K$ is
defined as $\varepsilon_-\tensor \id \tensor \varepsilon_+$ in this
triangular factorisation. 

Furthermore, write $L=\n_- \Cl(\g)\n_+ \subset \Cl(\g)$. Then 
$\Cl(\g)^\h \subseteq 1\tensor \Cl(\h)
\tensor 1 \dirsum L$. The Lemma follows immediately from the fact that 
$L\cdot \Cl(\g)$, $\Cl(\g) \cdot L$ are in the kernel of~$\K$. 
\end{proof}

\subsection{The r-matrix formula for $\K$}

Let us now express  
the Harish\dash Chandra map $\K\colon \Cl(\g)\to\Cl(\h)$ 
in terms of the projection $\K_0\colon \exterior\g\to\exterior\h$.  
For $i=1,\dots,n$, let 
$x_i$ (respectively $y_i$) be the positive (respectively negative)
root vectors in $\g$, normalized so that $(x_i,y_i)=1$. 
Denote 
$$
    \rmatr = \sum_{i=1}^n x_i \wedge y_i \qquad \in\exterior^2\g.
$$
The formula for $\rmatr$ is a standard way to write a classical
skew\dash symmetric r\dash matrix of $\g$. Now introduce the operator
$$
    \iota(\rmatr)\in \End(\exterior\g), \quad 
    \iota(\rmatr)= \sum_{i=1}^n \iota(x_i) \iota(y_i),
$$
of degree $-2$ with respect to the grading on $\exterior\g$. 
This operator is used in the following
%
\begin{proposition}
\label{prop:rmatr}
Modulo the identification of the spaces $\Cl(\g)$ and $\exterior\g$,
respectively $\Cl(\h)$ and $\exterior\h$,
$$\K(u)=\K_0(e^{\iota(\rmatr)}u)
$$ 
for any $u\in\exterior\g$. 
\end{proposition}
\begin{proof}
The algebra $\exterior \h$ is viewed as an exterior and Clifford
subalgebra of $\exterior \g$. For any $u\in \exterior\h$ one has 
$\K(u)=\K_0(u)=u$, and $\iota(\rmatr)u=0$ so that
$e^{\iota(\rmatr)}u=u$. Thus, both sides of the equation agree on
$u\in \Cl(\h) \subset \Cl(\g)$.

Observe, as in the proof of Lemma~\ref{lem:homo}, that 
$\Cl(\g)^\h\subset \Cl(\h)+ \n_-\cdot \Cl(\g)$. Let us show that
$\K(u)=\K_0(e^{\iota(\rmatr)}u)=0$ when $u\in \n_-\cdot \Cl(\g)$. 
Of course, $\K(u)=0$ simply by definition of the map $\K$. 
Now, we may assume that $u\in y_j \cdot \Cl(\g)$ for some $j$ between
$1$ and $n$; but since $\iota(\rmatr)$ does not depend on a particular
ordering of positive roots, we may assume $j$ to be $1$, i.e.,
$u=\gamma(y_1)u'$ for some $u'\in\exterior\g$. Here
$\gamma(y)=y\wedge\cdot + \iota(y)$ is the
operator of left Clifford multiplication by $y$ as in
\ref{subsect:sigma}.   

Note that the operators $\iota(x_i)\iota(y_i)$, $i=1,\dots,n$, 
pairwise commute and square
to zero. This follows from the fact that $\iota(x)$ and $\iota(y)$
anticommute for all $x,y\in \g$, see the proof of
Lemma~\ref{lem:superderiv}. 
Hence, we may write $e^{\iota(\rmatr)}$ as $\prod_{i=1}^n
e^{\iota(x_i)\iota(y_i)} =
\prod_{i=1}^n(1+\iota(x_i)\iota(y_i))$. Furthermore, because
$(x_i,y_1)=(y_i,y_1)=0$ for $i\ne 1$, it follows from
Lemma~\ref{lem:superderiv} that $\iota(x_i)$ commutes with
$\gamma(y_1)$ for $i\ne 1$. Thus, 
$$
\K_0(e^{\iota(\rmatr)} u) = \K_0 \bigl( 
(1+\iota(x_1)\iota(y_1))\gamma(y_1) u''\bigr)
$$
for some $u''\in\exterior\g$. It remains to note that 
$$
  (1+\iota(x_1)\iota(y_1))z = \gamma(x_1)\gamma(y_1)z-x_1\wedge
  y_1\wedge z 
$$
and that, obviously, $\K_0(x_1\wedge y_1\wedge z)=0$ for all
$z\in\exterior\g$. We are left with 
$$
\K_0(e^{\iota(\rmatr)} u) = \K_0 ( \gamma(x_1)\gamma(y_1)^2 u'') 
$$ 
which is zero since $y_1^2=(y_1,y_1)=0$ in the Clifford algebra $\Cl(\g)$. 

We have shown that both sides of the required equation agree when
$u\in\Cl(g)^\h$. Now suppose that $u$ is an eigenvector of non\dash zero weight for
the adjoint action of $\h$ on $\exterior\g$.  Since the maps $\K$,
$\K_0$ and $\iota(\rmatr)$ preserve the weight with respect
to the $\h$\dash action, $\K(u)$ and
$\K_0(e^{\iota(\rmatr)}u)$ must have, in $\exterior \h$, the same
non\dash zero weight as $u$. Since the adjoint action of $\h$ on
$\exterior\h$ is trivial, the latter is only possible if
$\K(u)=\K_0(e^{\iota(\rmatr)}u)=0$. The Proposition is proved.
\end{proof}

\begin{remark}
The map $p_G^T\circ \mathcal{T}$ in \cite[3.1]{AMW} coincides with 
$\K_0\circ e^{\iota(\rmatr)}$ (in our notation). 
Proposition~\ref{prop:rmatr} thus implies 
that the map $p_G^T\circ \mathcal{T}$ from \cite{AMW} is the same as our Harish\dash Chandra
map $\K$.  
\end{remark}

\begin{remark}
\label{rem:Khbar}
Recall that for each non\dash zero value of the deformation parameter $\hbar$
we can equip $\exterior\g$ with the structure of Clifford algebra
$\Cl_\hbar(\g)$. The latter Clifford algebra is built with respect to the
bilinear form $(\,,\,)_\hbar = \hbar\cdot (\,,\,)$ on $\g$ which is $\ad$\dash
invariant and non\dash degenerate. In particular, the Harish\dash Chandra map 
$$
\K_\hbar\colon \Cl_\hbar(\g)\to\Cl_\hbar(\h)
$$
is defined. Let us apply Proposition~\ref{prop:rmatr} to the
Clifford algebra $\Cl_\hbar(\g)$. 

Namely, if $x_i$, $y_i$ was a
positive/negative root vector  pair normalised by $(x_i,y_i)=1$ , 
then $x_i$, $\hbar^{-1} y_i$ will be such pair for the form
$ (\,,\,)_\hbar$. It follows that the classical r\dash matrix of $\g$
corresponding to the new form $(\,,\,)_\hbar$ is given by
$\rmatr_\hbar = \hbar^{-1}\rmatr$. Furthermore, the contraction operators on
$\exterior\g$ with respect to the bilinear form $(\,,\,)_\hbar$ are
given by $\iota_\hbar(x)=\hbar\,\iota(x)$. Hence we have a new operator 
$$
    \iota_\hbar(\rmatr_\hbar) := \sum_{i=1}^n
    \iota_\hbar(x_i)\iota_\hbar(\hbar^{-1}y_i) = \hbar\, \iota(\rmatr),
$$
which leads to the following corollary of Proposition~\ref{prop:rmatr}:
\begin{corollary}
\label{cor:rmatr}
$   \K_\hbar(u) = \K_0(e^{\hbar \iota(\rmatr)}u) = 
   \K_0(u) + \hbar\, \K_0(\iota(\rmatr) u) + 
    \frac{\hbar^2}{2!} \K_0(\iota(\rmatr)^2 u) + \dots
$.
\qed
\end{corollary}
The map $\K_\hbar$ is thus given  as a deformation of the Chevalley
projection $\K_0$. We emphasise that $\K_\hbar$ is injective on the
space $\J$ of invariants for
$\hbar\ne 0$ while $\K_0=\K_\hbar|_{\hbar^{\vphantom{H^2}}=0}$ is not.
\end{remark}

%
%
%
%

\section{The algebras of invariants
and Kostant's $\rho$-decomposition of $\Cl(\g)$}

This section contains the information on the
structure of the algebras $\J_S=\Sym(\g)^\g$,
$\J=\Cl(\g)^\g$ and $\Cl(\g)$ 
which will be used in the proof of our main results. 
We will recall several theorems of Kostant from \cite{Ko} where it is
assumed that $(\,,\,)$ is the Killing form on $\g$; however, they are
easily seen to hold when $(\,,\,)$ is any non\dash degenerate 
$\ad$\dash invariant form.

\subsection{Generators of the algebra of symmetric invariants}
\label{chevalley_generators}

Recall from Section~\ref{sect:firstsect} that the Chevalley projection map 
establishes an algebra homomorphism between $\J_S=\Sym(\g)^\g$ and the algebra 
$\Sym(\h)^W$ of Weyl group invariants in the polynomial algebra $\Sym(\h)$.
From the invariant theory of reflection groups \cite{ST,C} 
it follows that $\J_S$ has $\rank=\mathit{rank}(\g)$
algebraically independent homogeneous generators
$f_1,f_2,\dots,f_\rank$. These are defined 
up to multiplication  by non\dash zero constants and modulo $(\J_S^+)^2$,
where $\J_S^+=\J_S\cap \oplus_{n>0} S^n(\g)$.   
We will denote by $P_S$ the linear span of some chosen
$f_1,f_2,\dots,f_\rank$. The space $P_S$ is, in general, 
not uniquely defined and depends on the choice of the $f_i$.


Define the positive integers $m_1,\dots,m_\rank$ by $\deg f_i = m_i+1$.
The numbers $m_i$ are independent (up to reordering) of the choice   
of a particular set of the $f_i$ and are called the exponents of
$\g$. 

\subsection{Primitive skew\dash symmetric invariants}

It turns out that there is a skew  symmetric counterpart, $P$, of
the space $P_S\subset J_S$. 
First of all, extend the $\ad$\dash invariant form $(\,,\,)$ from $\g$ to the whole
of $\exterior \g$ in a standard way: $\exterior ^m \g$ is orthogonal
to $\exterior^n \g$ unless $m=n$, and  
$$
   (x_1\wedge \dots \wedge x_n, y_1 \wedge \dots \wedge y_n)=
    \det((x_i,y_j))_{i,j=1}^n
$$
where $x_i,y_i\in \g$. 
The space $P\subset \exterior\g$ of primitive alternating invariants 
is defined as the $(\,,\,)$\dash orthocomplement of $\J^+\wedge
\J^+$ in $\J^+$, where $\J^+$ 
is the augmentation ideal $\J\cap
(\sum_{m>0}\exterior^m\g)$. By a theorem of Koszul (see \cite[Theorem
26]{Ko}), the restriction of  $(\,,\,)$  to $\J$ and to $\J^+$ is
non\dash singular. Hence $P$ is a graded subspace of $\J$. 
Moreover, the dimension of $P$ is equal to the rank $\rank$ of
$\g$, and one actually knows the degrees where the graded components
of $P$ are located:
$$
P = \mathrm{span}\,\{p_1,p_2,\dots,p_\rank\},
\qquad
p_i\in (\exterior^{2m_i+1}\g)^\g,
$$
where $m_i$, as before, are the exponents of $\g$.

Under a natural bijection between $\J$ and the homology $H_*(\g)$ of
$\g$,  the elements of $P$ correspond to what is known as
primitive homology classes (see \cite[4.3]{Ko}).

%
%

Now it turns out that the space $\J$ 
of invariants, with the product induced from
$\exterior\g$, is itself an exterior algebra. 
The Hopf-Koszul\dash Samelson theorem (see \cite[4.3]{Ko} which refers
to \cite[Theorem 10.2]{Kz}) 
asserts that 
$$
\J=\exterior P,
$$
meaning that the map $\zeta\colon \exterior P\to \J$, 
which is an algebra homomorphism extending
the inclusion map $P\hookrightarrow \J$ (the elements of
$P$ anticommute in $J$, being of odd degree), is an
isomorphism.  

\subsection{The subalgebra $\J\subset \Cl(\g)$}
\label{clif_p}

Even more surprising is the result, due to Kostant, 
that the subalgebra $\J = \Cl(\g)^\g$ 
of the Clifford algebra $\Cl(\g)$, is itself a Clifford algebra,
generated by $P$ as primitive generators. 


Of course, the assertion $\J=\Cl(P)$ has a chance to be valid 
only if a primitive tensor $p\in P$, being
Clifford\dash squared, yields a constant. And this is indeed true; 
as shown in \cite[Theorem B]{Ko}, primitive skew\dash symmetric invariants
behave under Clifford multiplication as if they were elements of
degree $1$:
$$
               p\cdot p = (\alpha(p),p) \qquad\qquad
\text{for $p\in P$,}
$$
where $\alpha$ is defined as multiplication by the constant $(-1)^m$ 
on $P\cap \exterior^{2m+1}\g$. 
A further result of Kostant \cite[Theorem 35]{Ko} asserts that the 
 map $\zeta_{\Cl}\colon \Cl(P)\to \J$, 
which extends, as an algebra homomorphism, the inclusion map 
$P\hookrightarrow \J$, is an isomorphism of 
algebras. Here $\Cl(P)$ is the Clifford algebra of the space $P$
equipped with the non\dash degenerate bilinear form
$(p,q)_0=(\alpha(p),q)$. 
For later use, we record this here as theorem.  
\begin{theorem}[Kostant]
\label{thm:kostant}
In the above notation, 
the diagram
$$
\begin{CD}
\exterior P & @>\zeta>{}^\sim> & J \\
@V{\beta_{\wedge,P}}VV & & @VV{\beta_{\wedge,\g}}V \\
\Cl(P) & @>\zeta_{\Cl}>{}^\sim> & J
\end{CD}
$$
commutes. Here $\beta_{\wedge,P}\colon \exterior P \to \Cl(P)$
and $\beta_{\wedge,\g}\colon \exterior\g\to \Cl(\g)$ are skew\dash
symmetrisation maps for the respective exterior algebras. 
\end{theorem} 

\subsection{The Chevalley transgression map}
\label{subsect:t}

%
We briefly recall the useful transgression map, following
\cite[Section 6]{Ko}.
%
%
Let 
$$
    d\colon \g\to\exterior^2\g, \qquad dx = \frac{1}{2}\sum e_a\wedge [e^a,x]
$$
be the coboundary map for $\g$. (Its extension to $\exterior \g$ as a
derivation of degree $1$ is the coboundary in the standard Koszul
complex for $\g$.) 
Here $\{e_a\}$, $\{e^a\}$ is any pair of dual bases of $\g$ with respect
to the form $(\,,\,)$. 
Now introduce the algebra homomorphism
$$
    s\colon S(\g) \to \exterior^{\mathit{even}} \g, 
    \qquad
    s(x_1x_2\dots x_n) = dx_1 \wedge dx_2 \wedge \dots \wedge dx_n
$$
between $S(\g)$ and the commutative subalgebra
$\exterior^{\mathit{even}} \g = \sum_n\exterior^{2n} \g$ of $\exterior \g$. 

Denote by $\iota_S(x)f$ the directional derivative of $f\in S(\g)$ with
respect to $x\in\g$ (attention: $\g$ is identified with its dual space
$\g^*$ via the form $(\,,\,)$). In other words, 
$\iota_S(x)$ is the derivation of $S(\g)$ of degree $-1$ such that 
for $y$ in $S^1(\g)=\g$, one has 
$\iota_S(x)y=(x,y)\in S^0(\g)$.
The Chevalley transgression map may now be defined by the formula
$$
t(f) = \frac{(m!)^2}{(2m+1)!} \sum_a e_a\wedge s(\iota_S(e^a)f)
$$
due to Kostant \cite[Theorem 64]{Ko}.
This maps symmetric tensors of degree $m+1$ to alternating tensors of
degree $2m+1$.
By a result of Chevalley (see \cite[Theorem 66]{Ko}), 
for any choice of the space $P_S$ of
primitive symmetric invariants, $t\colon P_S\to P$ is a linear
isomorphism. (One observes that $t$ vanishes on $(J_S^+)^2$.)
One can choose $f_i\in P_S\cap S^{m_i+1}(\g)$
so that $t(f_i)=p_i$.

\subsection{The $\rho$-decomposition of $\Cl(\g)$}

Kostant's ``separation of variables'' result for the Clifford algebra 
\cite{Ko} asserts
that $\Cl(\g)$ is a free module over its
subalgebra $\J$.   
There is a subalgebra $E\subset\Cl(\g)$, which is in fact the Clifford 
centraliser of $\J$ in  $\Cl(\g)$, so that the Clifford algebra
factorises as 
$$
\Cl(\g) = E \tensor \J,
$$
where $\tensor$ is realised by Clifford multiplication. 
Moreover, $E$ can also be described as follows.
For $x\in \g$, denote 
$$
\delta(x)=\frac{1}{4} \sum_a e_a\cdot [e^a, x]\in \Cl(\g),
$$
where, as usual, $\{e_a\}$, $\{e^a\}$ are a pair of dual bases of
$\g$. It is easy to show that in fact, identifying the spaces
$\exterior\g$ and $\Cl(\g)$, one has 
$\delta(x)=\frac{1}{2}dx$ where 
$dx\in\exterior^2\g$ is the coboundary of $x$ introduced earlier. 
The equation 
$\delta([x,y])=\delta(x)\delta(y)-\delta(y)\delta(x)$ holds in the
Clifford algebra $\Cl(\g)$; see \cite[Proposition 28]{Ko}. 
Therefore, $\delta$ extends to a homomorphism 
$$
      \delta\colon U(\g) \to \Cl(\g)
$$
of associative algebras. One has $E= \delta(U(\g))$.
%
%
%

Now, let $\rho$ be the half sum of positive roots of $\g$, and let
$V_\rho$ denote the irreducible $\g$\dash module with highest
weight $\rho$. Kostant identifies $E$, as an algebra and a $\g$\dash
module, with the matrix algebra $\End V_\rho$, which is why the
factorisation $\Cl(\g)=E\tensor \J$ is referred to as the
$\rho$\dash decomposition of $\Cl(\g)$. Indeed, 
let $x_1,\dots,x_n$ be an ordering of positive root vectors in $\g$,
and let 
$\mu_+=x_1\cdot\dots\cdot x_n$ in $\Cl(\g)$.
Let $z\in U(\g)$ act on the subspace $E\mu_+$ of $\Cl(\g)$ via left
multiplication by $\delta(z)$. This makes $E\mu_+$ into a $U(\g)$\dash
module isomorphic to $V_\rho$, with highest weight vector $\mu_+$. 
The action of $E$ on $V_\rho\cong E\mu_+$ by left multiplication
induces a homomorphism $E\to \End V_\rho$. 
It turns out to be an isomorphism \cite[Theorem 40]{Ko}. 
 
%


\section{$\K$ is an isomorphism between $\Cl(\g)^\g$ and $\Cl(\h)$}
\label{sect:iso}

\subsection{The first main theorem}

In this section we prove our first main result about the Clifford
algebra analogue, 
$\K\colon \Cl(\g)\to \Cl(\h)$, of the Harish\dash Chandra map.
We will use the notation introduced in previous sections. 

\begin{theorem}
\label{thm:main1}
The restriction of the Harish\dash Chandra map $\K$ to the $\g$\dash
invariants in $\Cl(\g)$ is a 
superalgebra isomorphism $\K\colon \Cl(\g)^\g \to \Cl(\h)$.
\end{theorem}
\begin{proof}
We use Kostant's $\rho$\dash decomposition, $\Cl(\g)=E \tensor
\J$, of the Clifford algebra $\Cl(\g)$.
Let us denote by $E^\h$ the subalgebra $\{u\in E \mid \theta(\h)u=0\}$
of $\Cl(\g)^\h$. Then we have the tensor factorisation 
$$
    \Cl(\g)^\h = E^\h \tensor \J. 
$$
Recall from the previous Section that $E$ is the image of $U(\g)$
under the algebra map $\delta\colon U(\g)\to \Cl(\g)$. 
As $\delta$ is also a $\g$\dash module map, $E^\h =\delta(U(\g)^\h)$,
where $U(\g)^\h$ is $\ad \h$\dash invariants in $U(\g)$. 
Choose a basis of $U(\g)^\h$ consisting of monomials 
$y_{a_1}\dots y_{a_k}h_{b_1}\dots h_{b_l} x_{c_1}\dots x_{c_m}$, where
$y_a$ (resp.\ $x_c$) are negative (resp.\ positive) root vectors in
$\g$, $h_b$ are vectors in the Cartan subalgebra, and the product is,
of course, in $U(\g)$.
To see what the $\delta$\dash image of such a monomial can be, we use
the following

\begin{lemma}
\label{lem:calc}
$(i)$
For any $h\in\h$, $\K(\delta(h))$ is equal to the constant $\rho(h)$,
where $\rho\in \h^*$ is half the sum of positive roots of $\g$. 

$(ii)$
For any $x\in\n^+\subset\g$, $\delta(x)$ is in $\Cl(\g)\cdot \n^+$.
\end{lemma}
\begin{proof}[Proof of the lemma]
Choose a pair of dual bases of $\g$ in the following special way.
Let $x_1,\dots,x_n$ be positive root vectors in $\g$, 
corresponding to the positive roots $\beta_1,\dots,\beta_n$ and 
constituting the
basis of $\n^+$. Let $y_1,\dots,y_n$ be negative root vectors so that 
$(x_i,y_i)=1$, and let $h_1,\dots,h_\rank$ be some basis of
$\h$, orthonormal with respect to $(\,,\,)$.
The bases $x_1,\dots,x_n;h_1,\dots,h_\rank;y_1,\dots,y_n$ and 
$y_1,\dots,y_n;h_1,\dots,h_\rank;x_1,\dots,x_n$ of $\g$ 
are dual with respect to the $\ad$\dash invariant non\dash degenerate 
form $(\,,\,)$.
By definition of $\delta$ and using the fact that $\h$ is an Abelian Lie
subalgebra of $\g$, one calculates 
$$
\delta(h)=\frac{1}{4}\sum_{i=1}^n (x_i\cdot [y_i,h] +y_i\cdot [x_i,h])
=\frac{1}{4}\sum_i \beta_i(h)(x_i y_i - y_i x_i).
$$
Observe that $x_i y_i - y_i x_i=2-2y_i x_i$ in the Clifford algebra, 
and $\K(y_i x_i)=0$ because, by definition of $\K$, the kernel of $\K$
contains $\n_-\cdot \Cl(\g)$ and $\Cl(\g)\cdot \n_+$.
Thus one obtains $\K(\delta(h))=\frac{1}{4}\sum_i \beta_i(h)\cdot 2 = 
\rho(h)$, establishing part $(i)$ of the Lemma.

Now for $x\in\n^+$, 
calculation of $\delta(x)$ with respect to the same special pair of
dual bases of $\g$ will yield an expression with  
the following terms: $x_i\cdot [y_i,x]$, $y_i\cdot [x_i,x]$ and
$h_j\cdot [h_j,x]$.  The latter two clearly belong to $\Cl(\g)\n_+$. 
Rewrite 
$x_i\cdot [y_i,x]$ as $-[y_i,x]\cdot x_i + 2(x_i,[y_i,x])$. 
Here $-[y_i,x]\cdot x_i$ is
again in $\Cl(\g)\n_+$, and $(x_i,[y_i,x])=([x_i,y_i],x)=0$ because 
$[x_i,y_i]\in \h$ and $x\in \n_+$. Thus $\delta(x)\in \Cl(\g)\cdot \n_+$. 
The Lemma is proved. 
\end{proof}

\begin{remark}
The 
proof of Lemma~\ref{lem:calc} is similar to 
\cite[Proposition 37, Lem\-ma 38 and Theorem 39]{Ko}.
These statements lead to a Clifford algebra 
realisation of the representation with highest weight $\rho$  of a
semisimple Lie algebra (Chevalley\dash Kostant construction). 
This construction can be generalised to central
extensions of the corresponding loop algebra and in particular for 
$\widehat{\mathfrak{sl}_2}$. 
See the paper \cite{J2} by Joseph.
One may realise the basic modules of 
$\widehat{\mathfrak{sl}_2}$, which is done by Greenstein and Joseph in
\cite{GJ} and has no analogue in the semisimple case. 
In the infinite dimensional case, the ordering of the
factors in the expression for $\delta(h)$ becomes crucial. 
\end{remark}

We now continue the proof of Theorem~\ref{thm:main1}.
Consider a typical monomial $y_{a_1}\dots y_{a_k}$ $h_{b_1}\dots
h_{b_l}$ $
x_{c_1}\dots x_{c_m}$ in $U(\g)^\h$ as above. 
If $m>0$, the $\delta$\dash image of this monomial lies in 
$\Cl(\g)\delta(x_{c_m})$, which is in $\Cl(\g)\n^+$ by
Lemma~\ref{lem:calc}. By definition of $\K$, $\Cl(\g)\n^+$ lies in the
kernel of $\K$, thus the $\K\circ\delta$\dash image of the monomial
is zero. 

If $m=0$, then $k=0$ because the monomial must have weight zero with
respect to the adjoint action of $\h$ on $U(\g)$.  
The $\delta$\dash image of the monomial $h_{b_1}\dots h_{b_l}$ is 
$\delta(h_{b_1}) \dots  \delta(h_{b_l})$.
By Lemma~\ref{lem:calc} one has
$\K(\delta(h_{b_1})\dots  \delta(h_{b_l}))
=\rho(h_{b_1})\dots \rho(h_{b_l})\in\C$.
Thus, we have shown that 
$$
	\K(E^\h)=\C.
$$
Because $\K\colon \Cl(\g)^\h\to \Cl(\h)$ is a superalgebra
homomorphism (by Lemma~\ref{lem:homo}) and is surjective
(coincides with the identity map on $\Cl(\h)\subset \Cl(\g)^\h$), we have
$\Cl(\h)=\K(E^h)\K(\J)$. It follows that $\Cl(\h)=\K(\J)$,
that is, $\K\colon \J\to \Cl(\h)$ is surjective, hence bijective
by comparison of dimensions (both are $2^\rank$). 
Theorem~\ref{thm:main1} is proved.
\end{proof}

\subsection{A formula for $\K_\hbar\circ\delta$}

Looking at the proofs of Lemma~\ref{lem:calc} and
Theorem~\ref{thm:main1}, we conclude that the calculations which have
been made lead to a formula for the map $\K\circ\delta\colon
U(\g) \to \C$. Namely, it is apparent that 
$\K(\delta(u))=\Psi(u)(\rho)$, where $\Psi\colon U(\g)\mapsto
\Sym(\h)$ is the classical Harish\dash Chandra map introduced in 
\ref{subsect:cHC}, and $\mbox{}\,\cdot\, (\rho)$ denotes the evaluation of an
element of $\Sym(\h)$, viewed as a polynomial function on the space
$\h^*$, at the point $\rho\in\h^*$. For later purposes we will need a
slightly more general version of this formula, which is nevertheless
established by a completely analogous calculation. Recall the
``deformed'' Harish\dash Chandra map $\K_\hbar\colon \Cl(\g) \to
\Cl(\h)$, introduced in Remark~\ref{rem:Khbar}.
\begin{lemma}
\label{lem:composition}
The map $\K_\hbar \circ \delta \colon U(\g) \to \C$ is given by 
$$
\K_\hbar(\delta(u))= \Psi(u)(\hbar\, \rho).
\qed
$$
\end{lemma}

\subsection{$\K$ identifies $P$ and $\h$}

In the proof of Theorem~\ref{thm:main1} we relied on the result, due to
Kostant, that $\Cl(\g)$ factorises as $E\tensor \J$. 
We are going to obtain more information about the Harish-Chandra map
$\Phi\colon \Cl(\g)^\g \to \Cl(\h)$ using Kostant's description of $\J$ as the
Clifford algebra $\Cl(P)$, where $P$ is the space of primitive
$\g$\dash invariants in $\Cl(\g)$ as in the previous Section. 
It is one of the key results of \cite{Ko} that
$$
                   \iota(x)p \in E\qquad \text{for}\ p\in P, \quad x\in
                   \g;
$$
see \cite[Theorem E]{Ko}. From this, we deduce our
\begin{proposition}
\label{prop:bij}
The restriction of the Harish\dash Chandra map $\K$ 
to the space $P\subset\J$ of primitive invariants 
is a bijective linear map between $P$ and the Cartan subalgebra $\h$. 
\end{proposition}
\begin{proof}
Take a primitive invariant $p\in P$. For any $h\in \h$, one has 
$\iota(h)p\in E$ by the above result of Kostant. 
Therefore, it follows from Lemma~\ref{lem:calc} that $\K(\iota(h)p)\in
\C$. More precisely, $\K(\iota(h)p)$ is in the one\dash dimensional
subspace $\C \cdot 1 = \exterior^0 \h$ of $\exterior\h$. 
Let us now use 
\begin{lemma}
\label{lem:phi_delta}
For any $h\in\h$ and $u\in\Cl(\g)^\h$, one has $\iota(h)\K(u)=\K(\iota(h)u)$.
\end{lemma}
\begin{proof}[Proof of Lemma~\ref{lem:phi_delta}]
We have already observed that $\Cl(\g)^\h$ decomposes as a direct sum 
$\Cl(\h)\dirsum$ $\Cl(\g)\n^+\cap \Cl(\g)^\h$. 
Because $\iota(h)$ is a superderivation of $\Cl(\g)$
(Lemma~\ref{lem:superderiv}) and $\iota(h)\n^+=0$, 
$\iota(h)$ preserves the subspaces $\Cl(\h)$ and
$\Cl(\g)\n^+\cap \Cl(\g)^\h$ of $\Cl(\g)$, 
hence commutes with the projection onto $\Cl(\h)$. 
\end{proof}
Let $\K(p)$ be some element $q\in
\exterior \h$. Applying Lemma~\ref{lem:phi_delta} we establish that
$\iota(h)q\in \exterior^0 \h$ for all $h\in\h$. The restriction of the
form $(\, ,\,)$ to $\h$ is non\dash degenerate, therefore
there exists $h^1\in \h=\exterior^1 \h$ such that $\iota(h)h^1 =
\iota(h)q$ for all $h\in \h$. The intersection of kernels of all
contraction operators $\iota(\h)$ in the exterior algebra $\exterior
\h$ is its zero degree part, $\exterior^0\h=\C$: this fact can be
checked directly and is a particular case of a Nichols algebra
property (see Remark~\ref{rem:braided} above and \cite[Criterion
  3.2]{B}). Thus, $q=h^1+h^0$ for some $h^0\in \exterior^0\h$. 

But $\K$ is a map of superspaces by Lemma~\ref{lem:homo}, and a
primitive invariant $p\in P$ is odd. Therefore, the even component
$h^0$ of $q$ is zero, thus $\K(p)\in \h$. 
Hence $\K(P)\subset\h$, and by injectivity of $\Phi$ on $\Cl(\g)^\g$
(Theorem~\ref{thm:main1}) and comparison of dimensions, $\K(P)=\h$.
Proposition~\ref{prop:bij} is proved.
\end{proof}

\subsection{$\K$ induces an isomorphism between $\J$ and
  $\exterior\h$} 

We finish this Section with an observation that the Harish\dash
Chandra map $\Phi$ respects not only Clifford but also exterior
multiplication on the space of $\g$\dash invariants in $\exterior\g$. 

\begin{corollary}
\label{cor:wedge}
The map $\K\colon (\exterior\g)^\g \to \exterior \h$ is an algebra
isomorphism. 
\end{corollary} 
\begin{remark}
The Corollary asserts that the restriction $\K$ to
$\J$ respects the wedge
multiplication. This is not trivial, because 
it is not readily seen from the construction
of $\K$ why $\K(a\wedge b)=\K(a)\wedge \K(b)$ when $a,b\in
J$. Recall that in terms of the wedge product, the map $\K$ is given
by the r-matrix formula (Proposition~\ref{prop:rmatr}), and  
note that in general, $\K(a\wedge b)\ne\K(a)\wedge \K(b)$ for 
$a,b\in(\exterior\g)^\h$.
\end{remark}
\begin{proof}[Proof of Corollary \ref{cor:wedge}]

We need to show that the wedge product in $(\exterior \g)^\g$ is
respected by the map, fully written as $\sigma_{\h} \circ
\K \circ \sigma^{-1}_{\g}$, where 
$\sigma_{\h}\colon \Cl(\h) \to \exterior\h$ and
$\sigma_{\g}\colon \Cl(\g)\to\exterior\g$ are maps identifying the
underlying linear space of a Clifford algebra with the corresponding
exterior algebra; see \ref{subsect:sigma}. But by Theorem~\ref{thm:kostant} it is enough to show that 
$\sigma_{\h} \circ \K \circ \sigma^{-1}_{P}$ respects the  
wedge product in $\exterior P$, where $P$ is the space of primitive
invariants in $\exterior\g$. Let $p_1,\dots,p_\rank$ be a basis of $P$
such that $p_i$ are 
homogeneous in $\exterior\g$ and are pairwise orthogonal with respect
to the form $(\,,\,)$ extended to $\exterior\g$. 
Then $p_i$ are also orthogonal with respect to the form
$(\alpha(\cdot),\cdot)$, which gives rise to the Clifford algebra
$\Cl(P)$ as in~\ref{clif_p}. It follows that the $p_i$ pairwise
anticommute  in $\Cl(P)$, hence also in $\Cl(\g)$. 

For each subset $I=\{i_1<\dots <i_k\}$ of $\{1,\dots,\rank\}$, 
denote $p_I=p_{i_1}\wedge \dots \wedge p_{i_k}$. The $2^\rank$
elements $p_I$ form a basis of $\exterior P$. 
By orthogonality of the $p_i$,
$p_I=\sigma_P(p_{i_1}p_{i_2}\dots p_{i_k})$; applying 
$\K$ gives $h_{i_1}h_{i_2}\dots h_{i_k}\in\Cl(\h)$, where
$h_i=\K(p_i)$. By Proposition~\ref{prop:bij}, $h_1,\dots,h_\rank$ form
a basis of the Cartan subalgebra $\h$. Moreover, this basis is
orthogonal with respect to the restriction of the  form $(\,,\,)$ on
$\h$, because the $h_i$ pairwise anticommute in the Clifford algebra
$\Cl(\h)$ being the images of the $p_i$. Therefore,
$\sigma_{\h}(h_{i_1}h_{i_2}\dots h_{i_k}) = h_{i_1}\wedge
h_{i_2}\wedge \dots\wedge h_{i_k}$, which may be denoted by
$h_I$. 

Thus, modulo the appropriate  identifications, the map $\K\colon
(\exterior\g)^\g\to\exterior\h$ is given on the basis $\{p_I\mid
I\subseteq \{1,\dots,\rank\}\}$ of $(\exterior\g)^\g$ by
$\K(p_I)=h_I$. It manifestly follows that $\K$ is an isomorphism of
exterior algebras.
\end{proof}



\section{The principal basis of the Cartan subalgebra}

In the previous Section we established that the Harish\dash Chandra
map $\K\colon \Cl(\g) \to \Cl(\h)$ restricts to a bijective linear map
between the space $P$ of primitive $\g$\dash invariants in $\Cl(\g)$
and the Cartan subalgebra $\h$ of $\g$. 
Recall that $P$ is  a graded subspace of the exterior
algebra $\exterior \g$ of $\g$. The linear bijection $\K\colon P \to
\h$ thus induces a grading on the Cartan subalgebra $\h$. 
We will now show that this grading coincides 
with one arising in a different context --- from the decomposition of
$\g^\vee$, the Langlands dual Lie algebra of $\g$, as a module over its
principal three\dash dimensional simple subalgebra.

\subsection{Principal three-dimensional simple subalgebras}

In this subsection, we briefly recall known facts about principal
three-dimensional simple 
subalgebras (principal TDS) of $\g$. Our standard reference for this
is \cite[\S5]{KoBetti}.

Let $\rank$ be the rank of a semisimple Lie algebra  $\g$.
A nilpotent element $e$ of $\g$ is called principal nilpotent (or regular
nilpotent), if the centraliser $\g^e$ of $e$ in $\g$ is of minimal
possible dimension, namely $\dim \g^e=\rank$. 
By the well-known Jacobson\dash Morozov theorem, any non\dash zero nilpotent
element $e\in \g$ can be included in an $\Sl_2$-triple  
$(e, h, f)$ of elements of $\g$, i.e., a triple that fulfils 
the relations $[h,e]=2e$, $[h,f]=-2f$, $[e,f]=h$. 
Such an $\Sl_2$-triple is called principal, if   
$e$ is principal nilpotent. The linear span of a principal
$\Sl_2$-triple in $\g$ is referred to as a principal three\dash dimensional
simple subalgebra (principal TDS) of $\g$. 

The following characterisation of principal TDS is due to Kostant:
\begin{lemma}
If $\mathfrak a$ is a principal TDS of
$\g$, then $\g$, viewed as an $\mathfrak a$\dash module via the
adjoint action, is a direct sum of precisely $\rank$ simple 
$\mathfrak a$\dash modules. 
\qed
\end{lemma}
If a subalgebra $\mathfrak a\cong\Sl_2$ of $\g$ is not principal,  
then $\g$ is a direct sum of strictly more than $\rank$ simple
$\mathfrak a$\dash modules. 

All principal TDS $\g$ are conjugate (with respect
to the action of the adjoint group of $\g$).
We will be interested in one particular principal TDS $\mathfrak a_0$ of $\g$,
discovered independently by Dynkin and de Siebenthal. To define $\mathfrak a_0$,  
fix a Cartan subalgebra $\h$ of $\g$ and recall that our chosen
$\ad$\dash invariant form $(\,,\,)$, restricted to $\h$, is non\dash
degenerate, hence one may 
identify $\h$ with its dual space $\h^*$. In particular, the root system of
$\g$ becomes a subset of $\h$; a root $\alpha$ and the corresponding
coroot $\alpha^\vee$ are related via $\alpha^\vee =
2\alpha/(\alpha,\alpha)$. 

Choose a set of simple roots $\alpha_1, \dots,\alpha_\rank \in \h$.  
Denote by $\rho^\vee$ the element of $\h$ defined by the condition
$(\alpha_i,\rho^\vee)=1$, $i=1,\dots,\rank$; this element is half the
sum of positive coroots of $\g$. 
Let $e_i$ be root vectors corresponding
to the roots $\alpha_i$, and $f_i$ be root vectors
corresponding to roots $-\alpha_i$ normalised so that $(e_i,f_i)=1$.
Observe that $[e_i,f_i]=\alpha_i\in \h$. Indeed, for arbitrary
$h\in\h$ one has
$([e_i,f_i],h)=(e_i,[f_i,h])=(e_i,\alpha_i(h)f_i)=(\alpha_i,h)$.  
Put
$$
  e_0=\sum_{i=1}^\rank e_i, 
\qquad
  h_0 = 2\rho^\vee, 
\qquad
  f_0 = \sum_{i=1}^\rank c_i f_i,
$$
where $c_i$ are the coefficients in the expansion $2\rho^\vee =
\sum_{i=1}^\rank c_i \alpha_i$. It is not difficult to show that
$(e_0,h_0,f_0)$ is an $\Sl_2$-triple (using the fact that
$(\alpha_i,2\rho^\vee)=2$ for all $i$, that $[e_i,f_j]=0$ for $i\ne
j$, and that $[e_i,f_i]=\alpha_i$). Moreover, $e_0$ is a principal
nilpotent element of $\g$. The principal TDS $\mathfrak a_0$ of $\g$
is the linear span of the triple $(e_0,h_0,f_0)$. 

\subsection{The principal grading on $\h$}  

Recall that under the adjoint action of a principal TDS $\mathfrak a$, the Lie
algebra $\g$
breaks down into a direct sum of $\rank=\mathit{rank}\ \g$ simple $\mathfrak a$\dash modules. 
The dimensions of these simple modules were determined by Kostant in
\cite{KoBetti}: writing $V_d$ for the $d$\dash dimensional
$\Sl_2$\dash module, one has 
$$
	\g \cong V_{2m_1+1} \dirsum \dots \dirsum V_{2m_\rank+1},
$$
where $m_i$, $i=1,2,\dots,\rank$, are the exponents of $\g$;
see~\ref{chevalley_generators}.  
Each $\Sl_2$\dash module $V_{2m_i+1}$, being of odd dimension, has a
one\dash dimensional zero weight subspace $V_{2m_i+1}^0$. We now turn to the
case $\mathfrak a=\mathfrak a_0$, the distinguished principal TDS
introduced above. 
Fix an isomorphism between $\g$ and  $\oplus_i V_{2m_i+1}$ and regard
each $V_{2m_i+1}$ as an $\mathfrak a_0$\dash submodule of $\g$.
Clearly, $\oplus_i V_{2m_i+1}^0$ is the centraliser,
$\g^{h_0}$, of $h_0=2\rho^\vee$ in $\g$. Since all root vectors in $\g$
are eigenvectors of $\ad h_0$ corresponding to non\dash zero eigenvalues,  
it follows that $\g^{h_0}=\h$, the Cartan subalgebra of $\g$ (i.e.,
$h_0$ is a regular semisimple element of $\g$). One has the direct sum
decomposition
$$
      \h =  V_{2m_1+1}^0 \dirsum \dots \dirsum V_{2m_\rank+1}^0.
$$
This decomposition of $\h$ is not canonical, because there is some
freedom in choosing the isomorphism between $\g$ and $\oplus_i
V_{2m_i+1}$. However, for each $d$, the $V_d$\dash primary component 
$\oplus\{ V_{2m_i+1} : 2m_i+1=d\}$ is canonically defined. There is
thus a grading 
$$
     \h=\oplus_d \h_d, \qquad \h_d = \h\cap \oplus\{ V_{2m_i+1} : 2m_i+1=d\}
$$
on the Cartan subalgebra $\h$, which only depends on the choice of a
set of simple roots of $\h$ in $\g$ (i.e., the choice of a Borel
subalgebra $\mathfrak b \supset \h$). Note that the element
$\rho^\vee\in\h$ and the grading do not depend on a particular
non\dash degenerate $\ad$\dash invariant form $(\,,\,)$.
We refer to this grading as the principal grading on $\h$. 

\begin{remark}
There is another grading on $\g$ associated to any $\Sl_2$\dash triple 
$(e,h,f)$ in $\g$, namely the eigenspace decomposition of $\ad
h$. Such gradings are referred to as Dynkin gradings in \cite{EK}. 
Note that the Dynkin grading arising from $(e_0,h_0,f_0)$ has $\h$ as
the degree zero subspace and breaks down each of the $V_{2m_i+1}$ into
a direct sum of $1$\dash dimensional graded subspaces, so it is, in a
sense, transversal to the principal grading.
\end{remark}

\subsection{Elementary properties of the principal grading}

Let us observe a few straightforward properties of the principal
grading on $\h$. 

\begin{lemma}
\label{lem:easy}
Let $\h=\oplus_d \h_d$ be the principal grading on the Cartan
subalgebra $\h$. 

$(a)$ Non\dash zero homogeneous components $\h_d$ occur only in degrees
   $d=2m+1$, where $m$ is an exponent of $\g$. The least such degree
   is $3$. The dimension
   of $\h_{2m+1}$ is the number of times $m$ occurs as an exponent of
   $\g$. 

$(b)$ $\dim \h_3$ is the number of connected components in the Dynkin
   diagram of $\g$. 

$(c)$ The element $\rho^\vee$ of $\h$ belongs to $\h_3$.

$(d)$ Homogeneous components of $\h$ of different degrees are orthogonal with
   respect to any $\ad$\dash invariant form $(\,,\,)$ on $\g$.
\end{lemma}
\begin{proof}
$(a)$ is apparent from the definition of the principal grading. 
To establish $(b)$, one observes that $1$ occurs as an
exponent of $\g$ as many times as is the number of summands in the
decomposition of $\g$ into a direct sum of simple Lie algebras, which
is the number of connected components of the Dynkin diagram of $\g$. 
(A simple Lie algebra has only one exponent equal to $1$, and
taking the direct sum of semisimple Lie algebras corresponds to taking
the union of the multisets of exponents.) 
Note that $\mathfrak a_0\subset \g$ is a three\dash dimensional 
$\mathfrak a_0$\dash submodule of $\g$, hence $\mathfrak a_0\cap \h=\C
\rho^\vee\subset \h_3$ and $(c)$ follows.
To prove $(d)$, take $h,k\in \h$ such that $h\in V_{2m+1}^0\subset
\g$ and $k\in V_{2n+1}^0\subset \g$, with $m<n$ two distinct exponents of $\g$. 
By the representation theory of $\Sl_2$, there is a highest weight
vector $x\in V_{2n+1}$ for $\mathfrak a_0$, such that 
$k=(\ad f_0)^{n}x$. Then $(h,k)=((\ad f_0)^{n}x,k)=(-1)^{n}(x,(\ad
f_0)^{n}k)$. But $(\ad f_0)^{n}k=0$,
thus $h$, $k$ are orthogonal. 
\end{proof}

\subsection{The principal basis of $\h$ in a simple Lie algebra $\g$}

When $\g$ is simple, we can say more about the principal
grading on $\h$.

\begin{lemma}
Let $\h$ be the Cartan subalgebra of a simple Lie algebra $\g$. 
If
$\g$ is not of type $D_\rank$ with $\rank$ even, then all non\dash zero homogeneous
components of the principal grading $\h=\oplus_d \h_d$ are one\dash
dimensional. 
If $\g$ is of type $D_\rank$ with $\rank$ even, then
$\dim\h_d=1$ for $d\ne \rank-1$, and $\dim\h_{\rank-1}=2$. 
\end{lemma}
\begin{proof}
According to \cite[Table IV]{Bou}, simple Lie algebras of types other
than $D_\rank$ with even $\rank$ have distinct exponents; in type
$D_{\rank}$ with even $\rank$ all exponents have multiplicity $1$
except $\rank/2-1$ which has multiplicity $2$. Apply
Lemma~\ref{lem:easy}$(a)$. 
\end{proof}

Thus, in all simple Lie algebras apart from $\mathfrak{so}(2\rank)$,
$\rank$ even, a choice, $\mathfrak b \supset\h$, of a Borel and a Cartan
subalgebra gives rise to a distinguished basis $h_1,h_2,\dots,h_\rank$
of $\h$. The vector $h_i$ is determined, up to a scalar factor, by the
condition $\h_{2m_i+1}=\C h_i$. This basis is orthogonal with respect
to the Killing form, and $h_1=\rho^\vee$. We call $h_1,\dots,h_\rank$ the
principal basis of $\h$. 

For $\g\cong\mathfrak{so}(2\rank)$, $\rank-2$ vectors of the principal
basis are determined up to a scalar factor, and there is freedom in
choosing the remaining pair of vectors which must span a given
two\dash dimensional subspace $\h_{\rank-1}$ of $\h$.
We still call any basis, obtained by this procedure, a principal basis
of $\h$.

\subsection{The grading on $\h$ induced from $\g^\vee$}

Recall that we have identified the Cartan subalgebra $\h$ with its
dual space $\h^*$ via the non\dash degenerate $\ad$\dash invariant
form $(\,,\,)$, so that both the roots and the
coroots of $\g$ lie in $\h$. Consider $\g^\vee$, the Langlands dual
Lie algebra to $\g$; that is, the complex semisimple Lie algebra with
a root system dual to that of $\g$.
The roots of $\g^\vee$ are the coroots of $\g$, so we will assume that
$\g^\vee$ and $\g$ share the same Cartan subalgebra $\h^\vee=\h$. 

The above definition of principal grading applies to $\g^\vee$. One
has, therefore, a grading 
$$
    \h = \oplus_d \h^\vee_d
$$
on the Cartan subalgebra, which is the principal grading  induced from
$\g^\vee$. This grading depends on a choice of simple coroots 
$\alpha_1^\vee,\dots,\alpha_\rank^\vee$ which is the same as a choice
$\alpha_1,\dots,\alpha_\rank$ of simple roots of $\g$.  

If $\g$ is a simple Lie algebra, $\g^\vee$ is also simple,
so in that case $\h$ has a dual principal basis induced from $\g^\vee$.

\subsection{Main result}

We are now ready to state and prove the final main result of the
paper, Theorem~\ref{thm:main2} which also solves a conjecture of
Kostant. Once again, recall that $P\subset (\exterior\g)^\g$ is the
space of primitive skew\dash symmetric invariants. It is graded by
degree in $\exterior \g$.  The space $P$ is also
viewed 
as a subspace of
$\Cl(\g)^\g$. We already know that under the Harish\dash Chandra map
$\K\colon \Cl(\g)^\g \xrightarrow{\sim} \Cl(\h)$, the space $P$ is
isomorphically 
mapped onto $\h$.

\begin{theorem}
\label{thm:main2}
The grading on the Cartan subalgebra $\h$ of $\g$, induced from the grading on
$P$ by degree via the map $\K\colon P\xrightarrow{\sim}\h$, coincides with the
principal grading $\h=\oplus_d \h^\vee_d$ on $\h$ 
in the Langlands dual Lie algebra $\g^\vee$. 
\end{theorem}
\begin{proof}
We are going to use a result due to Chevalley, cited in~\ref{subsect:t}, that
the primitive skew\dash symmetric invariants are transgressive; i.e., if $f_1,\dots,f_\rank$ are independent homogeneous
generators of $S(\g)^\g$, then $t(f_1), \dots, t(f_\rank)$ is a basis
of $P$. Let $m_1\le m_2\le \dots \le m_\rank$ be the exponents of $\g$
so that $f_i\in S^{m_i+1}(\g)$. Denote $p_i=t(f_i)$, thus
$p_i\in\exterior^{2m_i+1}\g$. 
Furthermore, let $z_1,\dots,z_\rank$ be any basis of the Cartan
subalgebra $\h$ orthonormal with respect to the non\dash degenerate
$\ad$\dash invariant form $(\,,\,)$.  
As $\K(p_i)$ is an element of $\h$ by
Proposition~\ref{prop:bij}, we may write 
\begin{align*}
    \K(p_i) & = \sum_{j=1}^\rank (\K(p_i),z_j) z_j 
            = \sum_{j=1}^\rank \iota(z_j)\K(p_i)\cdot z_j
            = \sum_{j=1}^\rank \K(\iota(z_j)p_i)\cdot z_j
\\
            & = \sum_{j=1}^\rank \K(\iota(z_j)t(f_i))\cdot z_j ,
\end{align*}
where we used Lemma~\ref{lem:phi_delta} to exchange the maps $\K$ and
$\iota(z_j)$. 
We now observe that further to being two examples of braided
derivations (remark \ref{rem:braided}), the operators 
$\iota_S(z)\colon S^{m+1}(\g)\to S^m(\g)$ and $\iota(z)\colon
\exterior^{2m+1}\g\to\exterior^{2m}\g$ are 
intertwined (up to a scalar factor) by
the Chevalley
transgression map $t$. One has
$$
     \iota(z)t(f) = \frac{(m!)^2}{(2m)!}s(\iota_S(z)f)\qquad \text{for
     all}\ f\in \J_S\cap 
     S^{m+1}(\g), \      z\in \g
$$
by \cite[Theorem 73]{Ko}, where $s\colon S(\g) \to \exterior \g$ is
the homomorphism introduced in~\ref{subsect:t}. Hence
$$
\K(p_i) = c \sum_{j=1}^\rank \K(s(\iota_S(z_j)f_i))\cdot z_j
$$
for some non\dash zero constant $c$. 

In the next calculation, we are going to use results from \cite{Ko}
which are obtained for the case of simple $\g$. We thus assume $\g$ to
be simple until further notice. 

Let us now assume that the independent homogeneous generators 
$f_1,\dots,f_\rank$ of the algebra $S(\g)^\g$ of symmetric invariants
are chosen in a particular way, so as to  
span the orthogonal complement to $(J_S^+)^2$ in $J_S^+$ (compare
with~\ref{chevalley_generators}). Here orthogonality is with respect
to a natural extension of the form $(\ ,\ )$ from $\g$ onto
$S(\g)$:  identify $S(\g)$ with
$S(\g^*)$ which is acting on $S(\g)$ by differential operators with
constant coefficients, i.e., to $f\in S(\g)$ there corresponds
$\partial_f\in S(\g^*)$; now put 
$(f,g)$ to be the evaluation of the polynomial function $\partial_f g$
at zero. Such choice of primitive symmetric invariants is due to
Dynkin. Denote by $P_D$ the span of $f_1,\dots,f_\rank$ where $\deg
f_i=m_i+1$, and let
$(P_D)_{(m)}= P_D\cap S^{\le m}(\g)$. 

The space $P_D$ is especially useful for us  because of the following property
\cite[Theorem 87]{Ko}:
there exist $u_1,\dots,u_\rank\in P_D$, such that $u_i-f_i\in
(P_D)_{(m_i)}$ and 
$$
        \delta_{u_i}(z) = \frac{1}{2^{m_i}}s(\iota_S(z)f_i)
$$
for any $z\in\g$. Here
$$
    \delta_u\colon \g \to \Cl(\g), \qquad 
    \delta_u(z) = \delta(\beta(\iota_S(z)u))
$$
is a $\g$\dash equivariant map from $\g$ to $\Cl(\g)$, defined in
\cite[6.12]{Ko}; we remind the reader that $\beta\colon \Sym(\g)\to
U(\g)$ is the PBW symmetrisation map. 
(It is clear from the definition of the map $\delta_u$
that the image of $\delta_u$ is in the subalgebra $E=\mathop{\mathrm{Im}} \delta$ of
$\Cl(\g)$.) 
In fact, the  $u_i$ are not homogeneous, and are of the form 
$u_i = f_i + \sum_{k<i} a_{ik} f_k$ for some coefficients $a_{ik}\in\C$.
 
We substitute this in the previous formula for $\K(p_i)$ to obtain
$$
\K(p_i) = c \sum_{j=1}^\rank \K(\delta_{u_i}(z_j)) z_j
$$
for some non\dash zero $c$. The expression for the scalar 
$\K(\delta_{u_i}(z_j))$ is given in Lemma~\ref{lem:composition}
(applied for $\hbar=1$):
$$
      \K(\delta_{u_i}(z_j)) = \Psi(\beta(\iota_S(z_j)u_i)) (\rho),
$$
where $\Psi\colon U(\g)\to \Sym(\h)$ is the Harish\dash Chandra map,
and $\mbox{}\,\cdot\, (\rho)$ 
means the evaluation of an element of $S(\h)$ at the
point $\rho$. Observe that $\iota_S(z_j)u_i$ is a non\dash homogeneous
element of $\Sym^{\le m_i}(\g)$, with top degree homogeneous component
equal to $\iota_S(z_j)f_i$. Moreover, 
$\Psi(\beta(\iota_S(z_j)u_i))$ is an element of $S^{\le m_i}(\h)$,
also in general non\dash homogeneous, 
with component of top degree $m_i$ equal to 
$\Psi_0(\iota_S(z_j)f_i)$, where $\Psi_0\colon \Sym(\g)\to\Sym(\h)$ is
the Chevalley projection map. 

We would like to obtain a more satisfactory expression for 
$\K(\delta_{u_i}(z_j))$. To do this, we introduce the deformation
parameter $\hbar$ into the picture. It is easy to run the above
argument for the Harish-Chandra map $\K_\hbar$ instead of $\K$,
obtaining 
$$
\K_\hbar(p_i) = c \sum_{j=1}^\rank \K_\hbar(\delta_{u_i}(z_j)) z_j.
$$
We can now apply Lemma~\ref{lem:composition} for general $\hbar\in
\C$, obtaining
$\K_\hbar(\delta_{u_i}(z_j)) = \Psi(\beta(\iota_S(z_j)u_i)) (\hbar\,\rho)$, 
which leads to 
$$
  \K_\hbar(p_i) = c \sum_{j=1}^\rank   \Psi(\beta(\iota_S(z_j)u_i))
  (\hbar\,\rho) z_j. 
$$
Now let us observe, crucially, that the map $\K_\hbar$ is given by its
expansion in terms of $\hbar$ in Corollary~\ref{cor:rmatr}:
$$
   \K_\hbar(p_i) = \sum_{s\ge 0}\hbar^s \K_0(\iota(\rmatr)^s p_i)
     \qquad \in \qquad \sum_{s\ge 0}\hbar^s \exterior^{2m_i+1-2s}\h.
$$
We know that $\K_\hbar(p_i)$ is an element of degree $1$, which means
that only the term containing $\hbar^{m_i}$ is non\dash zero. It follows
that $\K_\hbar(p_i)$ is homogeneous in $\hbar$ of degree $m_i$. But
this means that all terms of degree lower than $m_i$ in the 
$\Psi(\beta(\iota_S(z_j)u_i)) (\hbar\,\rho)$ can be dropped, as they
make no contribution. What remains is therefore
$$
\K_\hbar(p_i)=  c \sum_{j=1}^\rank 
\Psi_0(\iota_S(z_j)f_i)(\hbar\,\rho)z_j= c \sum_{j=1}^\rank
(\iota_S(z_j)f_i)(\hbar \rho)z_j
$$
(the Chevalley projection $\Psi_0$ is irrelevant, since we are
evaluating at a point in $\h^*$ anyway). Now we can put $\hbar=1$ and
rewrite the formula in the following very simple way:
$$
\K(p_i) = c \, \iota_S(\rho)^{m_i} f_i.
$$
Recall that this only holds when $f_i$ is a homogeneous element of
degree $m_i+1$ in the Dynkin space $P_D$ of
primitive symmetric invariants. The formula may no longer be true if
$P_D$ is replaced by some other space $P_S$ spanned by independent
homogeneous generators of $\Sym(\g)^\g$. 

The proof of the Theorem for simple Lie algebras is concluded with the
use of Lemma~\ref{lem:last} below, 
which is essentially analogous to the calculation carried out
in \cite{Rohr} but is included here for the sake of
completeness. To extend the result to semisimple Lie algebras, note
that if $\g\cong \g_1 \dirsum\dots\dirsum \g_l$ is a direct sum of
pairwise commuting simple Lie algebras, then $\g^\vee \cong
\g_1^\vee\dirsum \dots\dirsum \g_l^\vee$. 
The Cartan subalgebra $\h^\vee$
of $\g^\vee$ is the direct sum of Cartan subalgebras of the
$\g^\vee_i$, and other obvious choices made for the root system of
$\g^\vee$ ensure that the de Siebenthal-Dynkin canonical principal TDS of
$\g^\vee$ is the direct sum of principal TDS of each of $\g^\vee_i$. 
Thus, the $d$th degree of the principal grading in $\h^\vee$ is
spanned by $d$th vectors of principal bases of the $\g^\vee_i$. On the
other hand, the Clifford algebra $\Cl(\g)$ is a tensor product of
pairwise supercommuting Clifford algebras of the $\g_i$. All this 
ensures that the primitive $\g$\dash invariants $P$ in $\Cl(\g)$ are a direct
sum of spaces of primitive $\g_i$\dash 
invariants in $\Cl(\g_i)$, the homogeneous basis of each of which
projects, under the 
Harish\dash Chandra map, to the corresponding Langlands dual principal
basis in the Cartan subalgebra $\h_i$. The result for semisimple
algebras thus follows by a straightforward direct sum argument. 
\end{proof}
\begin{lemma}
\label{lem:last}
Let $b_1,\dots,b_\rank$ be algebraically 
independent homogeneous generators of the
algebra $\Sym(\g)^\g$ of symmetric invariants of a simple Lie algebra
$\g$ such that $\deg b_k = m_k+1$. 
Put $h_k=\iota_S(\rho)^{m_k} b_k$.
If the $h_k$ are non\dash zero and pairwise orthogonal with respect to the
Killing form, then $h_1\dots,h_\rank$ are a principal basis of $\h$
induced from the Langlands dual Lie algebra $\g^\vee$.
\end{lemma}
\begin{proof}
Let $W\subset \mathit{GL}(\h)$ be the Weyl group of $\g$, which is the
same as the Weyl group of $\g^\vee$. 
The Langlands dual principal basis of $\h$ is orthogonal with respect
to the Killing form on $\g^\vee$, which in restriction to $\h$ is the
unique, up to a scalar factor, $W$\dash invariant form on $\h$. Thus,
the Langlands dual principal basis is orthogonal with respect to the
Killing form on $\g$. It is therefore sufficient to show that $h_k$ is
in the kernel of $(\ad e_0^\vee)^{m_k+1}$ in the algebra $\g^\vee$, 
where $(e_0^\vee,2\rho,f_0^\vee)$
is the canonical principal $\Sl_2$-triple of $\g^\vee$. Here $m_k$ is the $k$th
exponent of $\g^\vee$, which is the same as the $k$th exponent of
$\g$. 

Putting $n=m_k+1$ and $b=b_k$, it is enough to prove that for
$b\in\Sym^n(\g)^\g$, the element $h=\iota_S(\rho)^{n-1}b$ of $\h$ is in
the kernel of $(\ad e_0^\vee)^n$ in the algebra $\g^\vee$.    
Let $\bar b=\Psi_0(b)$ denote the image of $b$ under the
Chevalley projection map $\Psi_0\colon S(\g)^\g\to S(\h)^W$. 
Clearly, $h=\iota_S(\rho)^{n-1}\bar b$. Now
regard $\h$ as the Cartan subalgebra of $\g^\vee$, and let
$\Psi^\vee_0\colon S(\g^\vee)^{\g^\vee}\to S(\h)^W$ 
be the Chevalley projection
for $\g^\vee$. Put $b^\vee = (\Psi_0^\vee)^{-1}(\bar b)$ so that
$b^\vee$ is the image of $b$ under the algebra isomorphism 
$$
             (\Psi_0^\vee)^{-1}\Psi_0\colon S(\g)^\g \xrightarrow{\sim} S(\g^\vee)^{\g^\vee},
$$
and write $h=\iota_S(\rho)^{n-1}b^\vee$. 
Using the formula $(\ad x)\iota_S(y)=\iota_S([x,y])+\iota_S(y)\ad
x$ where $x,y\in \g^\vee$ and both sides are operators on
$S(\g^\vee)$, we can now write
$$
(\ad e_0^\vee)^n \iota_S(\rho)^{n-1}b^\vee
= \sum
\frac{n!}{d_1!\dots d_n!}
\iota_S((\ad e_0^\vee)^{d_1}\rho)
\dots
\iota_S((\ad e_0^\vee)^{d_{n-1}}\rho)
(\ad e_0^\vee)^{d_n}b^\vee.
$$
The sum on the right is over all $n$tuples $d_1,\dots,d_n$ of
non\dash negative integers, such that $d_1+\dots+d_n=n$. 
Observe that for each such $n$tuple, either $d_i\ge 2$ for some $i<n$ so
that $(\ad e_0^\vee)^{d_i} \rho=0$ (remembering  the key $\Sl_2$ relation 
$(\ad e_0^\vee)^2\rho=0$); or else $d_n\ge 1$, so that $(\ad
e_0^\vee)^{d_n}b^\vee=0$ because $b^\vee$ is $\ad \g^\vee$\dash
invariant. It follows that the right-hand side of
the last equation is zero, as required.
\end{proof}

%
%


\end{document}